\documentclass{amsart}
\usepackage{amssymb}
\usepackage{amsfonts}
\usepackage{hyperref}

\newtheorem{theorem}{Theorem}
\theoremstyle{plain}

\newtheorem{corollary}{Corollary}

\newtheorem{definition}{Definition}

\newtheorem{lemma}{Lemma}

\newtheorem{remark}{Remark}

\numberwithin{equation}{section}
\input{tcilatex}

\begin{document}
\title[Recurrence relations involving the hypergeometric function]{%
Recurrence relations of coefficients involving hypergeometric function with
an application}
\author{Zhen-Hang Yang }
\address{State Grid Zhejiang Electric Power Company Research Institute,
Hangzhou, Zhejiang, 310014, Hangzhou, P. R. China}
\email{yzhkm@163.com}
\date{April 10, 2022}
\dedicatory{Dedicated to my mother}

\begin{abstract}
For $a,b,p\in \mathbb{R}$, $-c\notin \mathbb{N\cup }\left\{ 0\right\} $ and $%
\theta \in \left[ -1,1\right] $, let%
\begin{equation*}
U_{\theta }\left( x\right) =\left( 1-\theta x\right) ^{p}F\left(
a,b;c;x\right) =\sum_{n=0}^{\infty }u_{n}\left( \theta \right) x^{n}.
\end{equation*}%
In this paper, we prove that the coefficients $u_{n}\left( \theta \right) $
for $n\geq 0$ satisfies a 3-order recurrence relation. In particular, $%
u_{n}\left( 1\right) $ satisfies a 2-order recurrence relation. These offer
a new way to study for hypergeometric function. As an example, we present
the necessary and sufficient conditions such that a hypergeometric mean
value is Schur m-power convex or concave on $\mathbb{R}_{+}^{2}$.
\end{abstract}

\maketitle

\section{Introduction}

For real numbers $a,b$, and $c$ with $c\neq 0,-1,-2,...$, the Gaussian
hypergeometric function is defined by%
\begin{equation*}
F\left( a,b;c;x\right) =\sum_{n=0}^{\infty }\frac{\left( a\right) _{n}\left(
b\right) _{n}}{\left( c\right) _{n}}\frac{x^{n}}{n!}
\end{equation*}%
for $x\in \left( -1,1\right) $, where $\left( a\right) _{n}$ denotes
Pochhammer symbol defined by%
\begin{equation*}
\left( a\right) _{n}=a\left( a+1\right) \cdot \cdot \cdot \left(
a+n-1\right) =\frac{\Gamma \left( n+a\right) }{\Gamma \left( a\right) },
\end{equation*}%
for $n=1,2,...$, and $(a)_{0}=1$ for $a\neq 0$, here $\Gamma
(x)=\int_{0}^{\infty }t^{x-1}e^{-t}dt$ $(x>0)$ is the gamma function.

The hypergeometric function $F(a;b;c;x)$ has the simple differentiation
formulas%
\begin{equation}
F^{\prime }\left( a,b;c;x\right) =\frac{ab}{c}F\left( a+1,b+1;c+1;x\right) .
\label{df}
\end{equation}%
The behavior of the hypergeometric function near $x=1$ in the three cases $%
a+b<c$, $a+b=c$, and $a+b>c$, $a,b,c>0$, is given by%
\begin{equation}
\left\{ 
\begin{array}{lc}
\bigskip F\left( a,b;c;1\right) =\dfrac{\Gamma \left( c\right) \Gamma \left(
c-a-b\right) }{\Gamma \left( c-a\right) \Gamma \left( c-b\right) } & \text{%
if }c>a+b\text{,} \\ 
\bigskip F\left( a,b;c;x\right) =\dfrac{R\left( a,b\right) -\ln \left(
1-x\right) }{B\left( a,b\right) }+O\left( \left( 1-x\right) \ln \left(
1-x\right) \right) & \text{if }c=a+b\text{,} \\ 
F\left( a,b;c;x\right) =(1-x)^{c-a-b}F(c-a,c-b;c;x) & \text{if }c<a+b\text{,}%
\end{array}%
\right.  \label{F-near1}
\end{equation}%
where%
\begin{equation*}
B\left( z,w\right) =\frac{\Gamma \left( z\right) \Gamma \left( w\right) }{%
\Gamma \left( z+w\right) }\text{, \ }\func{Re}\left( z\right) >0\text{, }%
\func{Re}\left( w\right) >0
\end{equation*}%
is the classical beta function,%
\begin{equation}
R\equiv R\left( a,b\right) =-2\gamma -\psi \left( a\right) -\psi \left(
b\right) \text{,}  \label{R(a,b)}
\end{equation}%
here $\psi \left( z\right) =\Gamma ^{\prime }\left( z\right) /\Gamma \left(
z\right) $, $\func{Re}\left( z\right) >0$ is the psi function and $\gamma $
is the Euler-Mascheroni constant. It should be noted that $F\left(
a,b;a+b;x\right) $ is called zero-balanced function.

In 2018, Yang \cite[Lemma 2]{Yang-JMAA-467-2018} established two recurrence
relations of coefficients of $\left( r^{\prime }\right) ^{p}K\left( r\right) 
$ and $\left( r^{\prime }\right) ^{p}E\left( r\right) $, where and in what
follows, $r^{\prime }=\sqrt{1-r^{2}}$.

\noindent \textbf{Proposition Y1}. \emph{We have}%
\begin{equation}
\left( r^{\prime }\right) ^{p}K\left( r\right) =\frac{\pi }{2}%
\sum_{n=0}^{\infty }a_{n}r^{2n},  \label{f}
\end{equation}%
\emph{where }$a_{0}=1$\emph{, }$a_{1}=1/4-p/2$\emph{\ and for }$n\geq 2$%
\emph{,}%
\begin{equation}
a_{n}=\frac{8n^{2}-4\left( p+3\right) n+\left( 2p+5\right) }{4n^{2}}a_{n-1}-%
\frac{\left( p-2n+3\right) ^{2}}{4n^{2}}a_{n-2};  \label{an}
\end{equation}%
\begin{equation}
\left( r^{\prime }\right) ^{p}E\left( r\right) =\frac{\pi }{2}%
\sum_{n=0}^{\infty }b_{n}r^{2n},  \label{g}
\end{equation}%
\emph{where }$b_{0}=1,b_{1}=-p/2-1/4$\emph{\ and for }$n\geq 1$%
\begin{equation}
b_{n+1}=\frac{1}{4}\frac{8n^{2}-4pn-2p-1}{\left( n+1\right) ^{2}}b_{n}-\frac{%
1}{4}\frac{\left( 2n-p-1\right) \left( 2n-p-3\right) }{\left( n+1\right) ^{2}%
}b_{n-1}.  \label{bn}
\end{equation}

In 2021, Yand and Tian \cite[Proposition 3]{Yang-RACSAM-115-2021} presented
another recurrence relation of those coefficients of $\left( 1-x\right)
^{-q}F\left( -1/2,-1/2;2;x\right) $.

\noindent \textbf{Proposition Y2}. \emph{Let }$q\in \mathbb{R}$. \emph{We
have}%
\begin{equation}
\left( 1-x\right) ^{-q}F\left( -\frac{1}{2},-\frac{1}{2};2;x\right)
=\sum_{n=0}^{\infty }u_{n}x^{n}  \label{f1}
\end{equation}%
\emph{with }$u_{0}=1$\emph{, }$u_{1}=q-1/8$\emph{\ and for }$n\geq 1$\emph{,}%
\begin{equation}
u_{n+1}=\frac{2n^{2}+\left( 2q+1\right) n+2q-1/4}{\left( n+1\right) \left(
n+2\right) }u_{n}-\frac{\left( n+q-1/2\right) \left( n+q-3/2\right) }{\left(
n+1\right) \left( n+2\right) }u_{n-1}.  \label{un+1-rr}
\end{equation}

Inspired the above results, the aim of this paper is to establish more
general recurrence relation of coefficients $u_{n}$ given by%
\begin{equation*}
U_{\theta }\left( x\right) =\left( 1-\theta x\right) ^{p}F\left(
a,b;c;x\right) =\sum_{n=0}^{\infty }u_{n}\left( \theta \right) x^{n}.
\end{equation*}%
As consequences, we also present recurrence relations of coefficients $%
u_{n}\left( \theta \right) $ for $\theta =-1,1$ and $v_{n}$ given by%
\begin{equation*}
V\left( x\right) =\ln \left( 1-x\right) \times F\left( a,b;c;x\right)
=\sum_{n=0}^{\infty }v_{n}x^{n}.
\end{equation*}%
As an application, we obtain necessary and sufficient conditions for which a
hypergeometric mean is Schur m-power convex or concave on $\mathbb{R}%
_{+}^{2} $.

\section{Main results}

We start with the useful lemma.

\begin{lemma}
\label{L-dF,F+}Let%
\begin{equation*}
\begin{array}{ccccc}
F\equiv F\left( a,b;c,z\right) , &  & F_{a-}\equiv F\left( a-1,b;c,z\right) ,
&  & F_{a+}\equiv F\left( a+1,b;c,z\right)%
\end{array}%
.
\end{equation*}%
Then%
\begin{eqnarray}
\frac{dF}{dz} &=&\frac{\left( c-a\right) F_{a-}+\left( a-c+bz\right) F}{%
z\left( 1-z\right) },  \label{dF} \\
\frac{dF_{a-}}{dz} &=&\frac{a-1}{z}\left( F-F_{a-}\right)  \label{dF-}
\end{eqnarray}
\end{lemma}

\begin{proof}
It was shown in \cite[p. 51--52, (7) and (18)]{Rainville-SF-M-1960} that%
\begin{eqnarray}
z\frac{dF}{dz}+aF &=&aF_{a+},  \label{zdF} \\
z\frac{dF_{a-}}{dz} &=&\left( a-1\right) F-\left( a-1\right) F_{a-}.
\label{zdF-}
\end{eqnarray}%
These in combination with Gauss' contiguous function relation%
\begin{equation*}
\left( c-a\right) F_{a-}+\left( 2a-c-az+bz\right) F+a\left( z-1\right)
F_{a+}=0
\end{equation*}%
see \cite[p. 558, (15.2.10)]{Abramowitz-HMFFGMT-1972} give the desired
formulas.
\end{proof}

We now are in position to state and prove the general recurrence relation.

\begin{theorem}
\label{T-Ut}Let $a,b,p\in \mathbb{R}$, $-c\notin \mathbb{N\cup }\left\{
0\right\} $ and $\theta \in \left[ -1,1\right] $. Then we have%
\begin{equation}
U_{\theta }\left( x\right) =\left( 1-\theta x\right) ^{p}F\left(
a,b;c;x\right) =\sum_{n=0}^{\infty }u_{n}x^{n},  \label{Uth}
\end{equation}%
with $u_{0}=1$, $u_{1}=ab/c-p\theta $,%
\begin{equation}
u_{2}=\frac{1}{2}\theta ^{2}p\left( p-1\right) -\theta p\frac{ab}{c}+\frac{1%
}{2}\frac{ab\left( b+1\right) \left( a+1\right) }{c\left( c+1\right) },
\label{u2}
\end{equation}%
and for $n\geq 2$,%
\begin{equation}
u_{n+1}=\frac{\xi _{n,p,\theta }\left( a,b,c\right) }{\left( n+1\right)
\left( n+c\right) }u_{n}-\theta \dfrac{\eta _{n,p,\theta }\left(
a,b,c\right) }{\left( n+1\right) \left( n+c\right) }u_{n-1}+\theta ^{2}\frac{%
\lambda _{n,p}\left( a,b\right) }{\left( n+1\right) \left( n+c\right) }%
u_{n-2},  \label{un+1-n-2-rr}
\end{equation}%
where%
\begin{eqnarray*}
\xi _{n,p,\theta }\left( a,b,c\right) &=&\left( n+a\right) \left( n+b\right)
+\theta \left[ 2n^{2}-2n\left( p-c+1\right) -cp\right] , \\
\eta _{n,p,\theta }\left( a,b,c\right) &=&2n^{2}+2\left( a+b-p-2\right)
n-\left( a+b-1\right) p+2\left( a-1\right) \left( b-1\right) \\
&&+\theta \left( n-p-1\right) \left( n-p+c-2\right) , \\
\lambda _{n,p}\left( a,b\right) &=&\left( n+a-p-2\right) \left(
n+b-p-2\right) .
\end{eqnarray*}
\end{theorem}

\begin{proof}[Proof of Theorem \protect\ref{T-Ut}]
Let%
\begin{equation}
U_{\theta }^{\ast }\left( x\right) =\left( 1-\theta x\right) ^{p}F\left(
a-1,b;c;x\right) =\sum_{n=0}^{\infty }u_{n}^{\ast }x^{n}.  \label{Ut*}
\end{equation}

Differentiating for $U_{\theta }$ yields%
\begin{equation*}
-p\theta \left( 1-\theta x\right) ^{p-1}F+\left( 1-\theta x\right)
^{p}F^{\prime }=\sum_{n=0}^{\infty }nu_{n}x^{n-1}.
\end{equation*}%
Substituting (\ref{dF}) into the above equation gives%
\begin{equation*}
-p\theta \left( 1-\theta x\right) ^{p-1}F+\left( 1-\theta x\right) ^{p}\frac{%
\left( c-a\right) F_{a-}+\left( a-c+bx\right) F}{x\left( 1-x\right) }%
=\sum_{n=0}^{\infty }nu_{n}x^{n-1},
\end{equation*}%
then multiplying by $x\left( 1-x\right) \left( 1-\theta x\right) $ leads to%
\begin{equation}
\begin{array}{c}
-p\theta x\left( 1-x\right) U_{\theta }\left( x\right) +\left( c-a\right)
\left( 1-\theta x\right) U_{\theta }^{\ast }\left( x\right) +\left(
a-c+bx\right) \left( 1-\theta x\right) U_{\theta }\left( x\right) \bigskip
\\ 
=\left( 1-x\right) \left( 1-\theta x\right) \sum_{n=0}^{\infty
}nu_{n}x^{n}.\bigskip%
\end{array}
\label{Ut-Ut*}
\end{equation}%
Expanding in power series yields%
\begin{eqnarray*}
&&-p\theta \left( \sum_{n=1}^{\infty }u_{n-1}x^{n}-\sum_{n=2}^{\infty
}u_{n-2}x^{n}\right) +\left( c-a\right) \left( \sum_{n=0}^{\infty
}u_{n}^{\ast }x^{n}-\theta \sum_{n=1}^{\infty }u_{n-1}^{\ast }x^{n}\right) \\
&&+\left( a-c\right) \sum_{n=0}^{\infty }u_{n}x^{n}+\left( b-\theta \left(
a-c\right) \right) \sum_{n=1}^{\infty }u_{n-1}x^{n}-b\theta
\sum_{n=2}^{\infty }u_{n-2}x^{n} \\
&=&\sum_{n=1}^{\infty }nu_{n}x^{n}-\left( \theta +1\right)
\sum_{n=2}^{\infty }\left( n-1\right) u_{n-1}x^{n}+\theta \sum_{n=2}^{\infty
}\left( n-2\right) u_{n-2}x^{n},
\end{eqnarray*}%
which can be arranged as%
\begin{eqnarray*}
&&\left( c-a\right) \left( u_{0}^{\ast }-u_{0}\right) +\left[ -p\theta
u_{0}+\left( c-a\right) \left( u_{1}^{\ast }-\theta u_{0}^{\ast }\right)
+\left( a-c\right) u_{1}+\left( b-\theta \left( a-c\right) \right) u_{0}%
\right] x \\
&&+\sum_{n=2}^{\infty }\left[ \left( c-a\right) \left( u_{n}^{\ast }-\theta
u_{n-1}^{\ast }\right) +\left( a-c\right) u_{n}+\left( b-\theta \left(
a-c+p\right) \right) u_{n-1}+\left( p-b\right) \theta u_{n-2}\right] x^{n} \\
&=&u_{1}x+\sum_{n=2}^{\infty }\left[ nu_{n}-\left( \theta +1\right) \left(
n-1\right) u_{n-1}+\theta \left( n-2\right) u_{n-2}\right] x^{n}.
\end{eqnarray*}%
Comparing coefficients of $x^{n}$ gives $\left( c-a\right) \left(
u_{0}^{\ast }-u_{0}\right) =0$,%
\begin{equation*}
\left( c-a\right) \left( u_{1}^{\ast }-\theta u_{0}^{\ast }\right) +\left(
b-\theta \left( a-c+p\right) \right) u_{0}+\left( a-c\right) u_{1}=u_{1},
\end{equation*}%
\begin{eqnarray*}
&&\left( c-a\right) \left( u_{n}^{\ast }-\theta u_{n-1}^{\ast }\right)
+\left( p-b\right) \theta u_{n-2}+\left( b-\theta \left( a-c+p\right)
\right) u_{n-1}+\left( a-c\right) u_{n} \\
&=&\theta \left( n-2\right) u_{n-2}-\left( \theta +1\right) \left(
n-1\right) u_{n-1}+nu_{n}\text{ for }n\geq 2\text{,}
\end{eqnarray*}%
where the third relation can be changed to 
\begin{equation}
\begin{array}{r}
\left( c-a\right) \left( u_{n}^{\ast }-\theta u_{n-1}^{\ast }\right) =\left(
n-a+c\right) u_{n}+\theta \left( n-2+b-p\right) u_{n-2}\bigskip \\ 
+\left[ \theta \left( a-c+p\right) -b-\left( \theta +1\right) \left(
n-1\right) \right] u_{n-1}\text{ for }n\geq 2\text{.}\bigskip%
\end{array}
\label{un*-un-r1}
\end{equation}

In a similar way, differentiating for $U_{\theta }^{\ast }$ yields%
\begin{equation*}
-p\theta \left( 1-\theta x\right) ^{p-1}F_{a-}+\left( 1-\theta x\right)
^{p}F_{a-}^{\prime }=\sum_{n=0}^{\infty }nu_{n}^{\ast }x^{n-1}.
\end{equation*}%
Substituting (\ref{dF-}) into the above equation gives%
\begin{equation*}
-p\theta \left( 1-\theta x\right) ^{p-1}F_{a-}+\left( 1-\theta x\right) ^{p}%
\frac{a-1}{x}\left( F-F_{a-}\right) =\sum_{n=0}^{\infty }nu_{n}^{\ast
}x^{n-1}
\end{equation*}%
then multiplying by $x\left( 1-\theta x\right) $ lead to%
\begin{equation*}
-p\theta xU_{\theta }^{\ast }\left( x\right) +\left( a-1\right) \left(
1-\theta x\right) \left( U_{\theta }\left( x\right) -U_{\theta }^{\ast
}\left( x\right) \right) =\left( 1-\theta x\right) \sum_{n=0}^{\infty
}nu_{n}^{\ast }x^{n}.
\end{equation*}%
Expanding in power series yields%
\begin{equation*}
-p\theta \sum_{n=1}^{\infty }u_{n-1}^{\ast }x^{n}+\left( a-1\right) \left(
1-\theta x\right) \sum_{n=0}^{\infty }\left( u_{n}-u_{n}^{\ast }\right)
x^{n}=\left( 1-\theta x\right) \sum_{n=0}^{\infty }nu_{n}^{\ast }x^{n},
\end{equation*}%
which can be arranged as%
\begin{eqnarray*}
&&\left( a-1\right) \left( u_{0}-u_{0}^{\ast }\right) -p\theta
\sum_{n=1}^{\infty }u_{n-1}^{\ast }x^{n}+\left( a-1\right)
\sum_{n=1}^{\infty }\left( u_{n}-\theta u_{n-1}-u_{n}^{\ast }+\theta
u_{n-1}^{\ast }\right) x^{n} \\
&=&\sum_{n=1}^{\infty }\left( nu_{n}^{\ast }-\theta \left( n-1\right)
u_{n-1}^{\ast }\right) x^{n}.
\end{eqnarray*}%
Comparing coefficients of $x^{n}$ gives $\left( a-1\right) \left(
u_{0}-u_{0}^{\ast }\right) =0$ and%
\begin{equation*}
-p\theta u_{n-1}^{\ast }+\left( a-1\right) \left( u_{n}-\theta
u_{n-1}-u_{n}^{\ast }+\theta u_{n-1}^{\ast }\right) =nu_{n}^{\ast }-\theta
\left( n-1\right) u_{n-1}^{\ast }\text{ for }n\geq 1,
\end{equation*}%
where the second relation can be written as%
\begin{equation}
\left( n+a-1\right) u_{n}^{\ast }-\theta \left( n-2+a-p\right) u_{n-1}^{\ast
}=\left( a-1\right) u_{n}-\theta \left( a-1\right) u_{n-1}\text{ for }n\geq 1%
\text{.}  \label{un*-un-r2}
\end{equation}

Solving the Eqs. (\ref{un*-un-r1}) and (\ref{un*-un-r2}) for $u_{n-1}^{\ast
} $ with $n\geq 2$ gives%
\begin{equation}
\begin{array}{c}
\theta \left( p+1\right) \left( a-c\right) u_{n-1}^{\ast }=n\left(
n+c-1\right) u_{n}+\theta \left( n+a-1\right) \left( n+b-p-2\right)
u_{n-2}\bigskip \\ 
-\left[ \left( \theta +1\right) n^{2}+\left( a+b+\left( c-p-2\right) \theta
-2\right) n+\left( a-1\right) \left( b-\left( p+1\right) \theta -1\right) %
\right] u_{n-1}.\bigskip%
\end{array}
\label{un*-un-n-2}
\end{equation}%
When $\theta \left( p+1\right) \neq 0$, substituting the expressions of $%
\left( c-a\right) u_{n-1}^{\ast }$ and $\left( c-a\right) u_{n}^{\ast }$
into (\ref{un*-un-r1}) and arranging give%
\begin{equation}
\begin{array}{l}
u_{n+1}-\dfrac{\left( 2\theta +1\right) n^{2}+\left( a+b+2\theta \left(
c-1-p\right) \right) n+\left( ab-cp\theta \right) }{\left( n+1\right) \left(
n+c\right) }u_{n}\bigskip \\ 
+\theta \dfrac{%
\begin{array}{c}
\left( \theta +2\right) n^{2}+\left( 2a+2b-2p+\theta \left( c-2p-3\right)
-4\right) n \\ 
+\theta \left( p+1\right) \left( p-c+2\right) +2ab-\left( p+2\right) \left(
a+b-1\right)%
\end{array}%
}{\left( n+1\right) \left( n+c\right) }u_{n-1}\bigskip \\ 
-\theta ^{2}\dfrac{\left( n+a-p-2\right) \left( n+b-p-2\right) }{\left(
n+1\right) \left( n+c\right) }u_{n-2}=0\text{ for }n\geq 2,\bigskip%
\end{array}
\label{un+1-n-2-rra}
\end{equation}%
which implies (\ref{un+1-n-2-rr}).

On the other hand, by the Cauchy product formula we have%
\begin{equation*}
u_{n}=\sum_{k=0}^{n}\frac{\left( a\right) _{k}\left( b\right) _{k}}{k!\left(
c\right) _{k}}\frac{\theta ^{n-k}\left( -p\right) _{n-k}}{\left( n-k\right) !%
},
\end{equation*}%
which yields $u_{0}=1$,%
\begin{equation*}
u_{1}=\frac{ab}{c}-p\theta \text{ \ and \ }u_{2}=\frac{1}{2}\theta
^{2}p\left( p-1\right) -\theta p\frac{ab}{c}+\frac{1}{2}\frac{ab\left(
b+1\right) \left( a+1\right) }{c\left( c+1\right) }.
\end{equation*}

When $\theta =0$, the recurrence relation (\ref{un+1-n-2-rr}) is obviously
true. When $p=-1$, we see that%
\begin{equation*}
u_{n}=\sum_{k=0}^{n}\frac{\left( a\right) _{k}\left( b\right) _{k}}{k!\left(
c\right) _{k}}\theta ^{n-k},
\end{equation*}%
which satisfies%
\begin{equation*}
u_{n+1}-\theta u_{n}=\frac{\left( a\right) _{n+1}\left( b\right) _{n+1}}{%
\left( n+1\right) !\left( c\right) _{n+1}}.
\end{equation*}%
This also satisfied the recurrence relation (\ref{un+1-n-2-rr}), and the
proof is done.
\end{proof}

Taking $\theta =-1$ in Theorem \ref{T-Ut} we have the following corollary.

\begin{corollary}
\label{C-th=-1}Let $a,b,p\in \mathbb{R}$, $-c\notin \mathbb{N\cup }\left\{
0\right\} $. Then we have%
\begin{equation*}
U_{-1}\left( x\right) =\left( 1+x\right) ^{p}F\left( a,b;c;x\right)
=\sum_{n=0}^{\infty }u_{n}x^{n},
\end{equation*}%
with $u_{0}=1$, $u_{1}=ab/c+p$,%
\begin{equation*}
u_{2}=\frac{1}{2}p\left( p-1\right) +p\frac{ab}{c}+\frac{1}{2}\frac{ab\left(
b+1\right) \left( a+1\right) }{c\left( c+1\right) },
\end{equation*}%
and for $n\geq 2$,%
\begin{equation*}
u_{n+1}=\frac{\xi _{n,p}\left( a,b,c\right) }{\left( n+1\right) \left(
n+c\right) }u_{n}+\dfrac{\eta _{n,p}\left( a,b,c\right) }{\left( n+1\right)
\left( n+c\right) }u_{n-1}+\frac{\lambda _{n,p}\left( a,b\right) }{\left(
n+1\right) \left( n+c\right) }u_{n-2},
\end{equation*}%
where%
\begin{eqnarray*}
\xi _{n,p}\left( a,b,c\right) &=&-n^{2}+\left( a+b-2c+2p+2\right) n+\left(
ab+cp\right) , \\
\eta _{n,p}\left( a,b,c\right) &=&\left( n+2a+2b-c\right) \left( n-1\right)
-p^{2}-\left( a+b-c+2\right) p+2ab, \\
\lambda _{n,p}\left( a,b\right) &=&\left( n+a-p-2\right) \left(
n+b-p-2\right) .
\end{eqnarray*}
\end{corollary}

Taking $\theta =1$ we can obtain a 3-order recurrence relation of
coefficients of $U_{1}\left( x\right) $. From the proof of Theorem \ref{T-Ut}%
, however, the order can be reduced from 3 to 2.

\begin{corollary}
\label{C-th=1}Let $a,b,p\in \mathbb{R}$, $-c\notin \mathbb{N\cup }\left\{
0\right\} $ . Then we have%
\begin{equation*}
U_{1}\left( x\right) =\left( 1-x\right) ^{p}F\left( a,b;c;x\right)
=\sum_{n=0}^{\infty }u_{n}x^{n},
\end{equation*}%
with $u_{0}=1$, $u_{1}=ab/c-p$ and for $n\geq 1$,%
\begin{equation}
u_{n+1}=2\alpha _{n}u_{n}-\beta _{n}u_{n-1},  \label{u-th=1-rr}
\end{equation}%
where%
\begin{eqnarray}
\alpha _{n} &=&\dfrac{1}{2}\dfrac{2n^{2}+\left( a+b+c-2p-1\right) n+ab-cp}{%
\left( n+1\right) \left( n+c\right) },  \label{aln} \\
\beta _{n} &=&\dfrac{\left( n+a-p-1\right) \left( n+b-p-1\right) }{\left(
n+1\right) \left( n+c\right) }.  \label{ben}
\end{eqnarray}
\end{corollary}

\begin{proof}[Proof of Corollary \protect\ref{C-th=1}]
When $\theta =1$, dividing $\left( 1-x\right) $ by the relation (\ref{Ut-Ut*}%
) gives%
\begin{equation}
-pxU_{1}\left( x\right) +\left( c-a\right) U_{1}^{\ast }\left( x\right)
+\left( a-c+bx\right) U_{1}\left( x\right) =\left( 1-x\right)
\sum_{n=0}^{\infty }nu_{n}x^{n}.
\end{equation}%
Expanding in power series yields%
\begin{eqnarray*}
&&-p\sum_{n=1}^{\infty }u_{n-1}x^{n}+\left( c-a\right) \sum_{n=0}^{\infty
}u_{n}^{\ast }x^{n}+\left( a-c\right) \sum_{n=0}^{\infty
}u_{n}x^{n}+b\sum_{n=1}^{\infty }u_{n-1}x^{n} \\
&=&\sum_{n=1}^{\infty }nu_{n}x^{n}-\sum_{n=1}^{\infty }\left( n-1\right)
u_{n-1}x^{n},
\end{eqnarray*}%
Comparing coefficients of $x^{n}$ gives $\left( c-a\right) \left(
u_{0}^{\ast }-u_{0}\right) =0$,%
\begin{equation*}
-pu_{n-1}+\left( c-a\right) u_{n}^{\ast }+\left( a-c\right)
u_{n}+bu_{n-1}=nu_{n}-\left( n-1\right) u_{n-1}\text{ for }n\geq 1\text{,}
\end{equation*}%
which implies%
\begin{equation}
\left( c-a\right) u_{n}^{\ast }=\left( n-a+c\right) u_{n}-\left(
n-1+b-p\right) u_{n-1}\text{ for }n\geq 1\text{.}  \label{un*-un,n-1}
\end{equation}%
When $c\neq a$, eliminating $u_{n}^{\ast }$ and $u_{n-1}^{\ast }$ from the
Eqs. (\ref{un*-un,n-1}) and (\ref{un*-un-r2}) for $\theta =1$ yields%
\begin{eqnarray*}
u_{n} &=&\frac{2\left( n-1\right) ^{2}+\left( a+b+c-2p-1\right) \left(
n-1\right) +ab-cp}{n\left( c+n-1\right) }u_{n-1} \\
&&-\frac{\left( n+a-p-2\right) \left( n+b-p-2\right) }{n\left( c+n-1\right) }%
u_{n-2}\text{ for }n\geq 2.
\end{eqnarray*}%
Replacing $n$ by $n+1$ give the desired recurrence formula (\ref{u-th=1-rr}).

When $c=a$, from the relation (\ref{un*-un,n-1}) we have%
\begin{equation*}
u_{n}=\frac{n-1+b-p}{n}u_{n-1}\text{ for }n\geq 1,
\end{equation*}%
which, by an easy check, satisfies the recurrence relation (\ref{u-th=1-rr}).

The values of $u_{0}$ and $u_{1}$ follows Theorem \ref{T-Ut}, and the proof
is done.
\end{proof}

\begin{corollary}
\label{C-U0}Let $a,b,p\in \mathbb{R}$, $-c\notin \mathbb{N\cup }\left\{
0\right\} $. Then we have%
\begin{equation}
V\left( x\right) =\ln \left( 1-x\right) \times F\left( a,b;c;x\right)
=\sum_{n=0}^{\infty }v_{n}x^{n},  \label{V}
\end{equation}%
with $v_{0}=0$, $v_{1}=-1$ and for $n\geq 1$,%
\begin{equation}
v_{n+1}=2\alpha _{n}v_{n}-\beta _{n}v_{n-1}+\gamma _{n}w_{n},
\label{vn+1-n-1-rr}
\end{equation}%
where $\alpha _{n}$, $\beta _{n}$ are defined in (\ref{aln}), (\ref{ben}),
and%
\begin{eqnarray*}
\gamma _{n} &=&\frac{\left( c-b-a\right) n^{2}+\left( a+b-2ab\right)
n-c\left( a-1\right) \left( b-1\right) }{\left( n+1\right) \left(
n+a-1\right) \left( n+b-1\right) \left( n+c\right) }, \\
w_{n} &=&\frac{\left( a\right) _{n}\left( b\right) _{n}}{n!\left( c\right)
_{n}}=\frac{\Gamma \left( c\right) }{\Gamma \left( a\right) \Gamma \left(
b\right) }\frac{\Gamma \left( n+a\right) \Gamma \left( n+b\right) }{n!\Gamma
\left( n+c\right) }.
\end{eqnarray*}
\end{corollary}

\begin{proof}[Proof of Corollary \protect\ref{C-U0}]
Notice that%
\begin{equation*}
F\left( a,b;c;x\right) =\sum_{n=0}^{\infty }w_{n}x^{n}\text{ \ and \ }%
\lim_{p\rightarrow 0}\frac{\left( 1-x\right) ^{p}-1}{p}=\ln \left(
1-x\right) ,
\end{equation*}%
we have%
\begin{equation*}
V\left( x\right) =\ln \left( 1-x\right) \times F\left( a,b;c;x\right)
=\lim_{p\rightarrow 0}\frac{\left( 1-x\right) ^{p}F\left( a,b;c;x\right)
-F\left( a,b;c;x\right) }{p},
\end{equation*}%
and then,%
\begin{equation*}
\sum_{n=0}^{\infty }v_{n}x^{n}=\lim_{p\rightarrow 0}\sum_{n=0}^{\infty }%
\frac{u_{n}-w_{n}}{p}x^{n}=\sum_{n=0}^{\infty }\left( \lim_{p\rightarrow 0}%
\frac{u_{n}-w_{n}}{p}\right) x^{n},
\end{equation*}%
which implies that%
\begin{equation*}
v_{n}=\lim_{p\rightarrow 0}\frac{u_{n}-w_{n}}{p}\text{ \ for }n\geq 0\text{.}
\end{equation*}%
Evidently, $v_{0}=0$, $v_{1}=-1$. To obtain (\ref{vn+1-n-1-rr}), we write
the recurrence relation (\ref{u-th=1-rr}) as%
\begin{equation}
\begin{array}{c}
\dfrac{u_{n+1}-w_{n+1}}{p}=2\alpha _{n}\dfrac{u_{n}-w_{n}}{p}-\beta _{n}%
\dfrac{u_{n-1}-w_{n-1}}{p}\bigskip \\ 
+\dfrac{2\alpha _{n}w_{n}-\beta _{n}w_{n-1}-w_{n+1}}{p}.\bigskip%
\end{array}
\label{un-wn/p-rr}
\end{equation}%
Letting $p\rightarrow 0$ and noting that%
\begin{equation*}
\frac{2\alpha _{n}w_{n}-\beta _{n}w_{n-1}-w_{n+1}}{p}=\frac{w_{n}}{p}\left[
2\alpha _{n}-\beta _{n}\frac{n\left( n+c-1\right) }{\left( n+a-1\right)
\left( n+b-1\right) }-\frac{\left( n+a\right) \left( n+b\right) }{\left(
n+1\right) \left( n+c\right) }\right]
\end{equation*}%
\begin{eqnarray*}
&=&\frac{\left( c-b-a-p\right) n^{2}+\left( a+b-2ab+p-cp\right) n-c\left(
a-1\right) \left( b-1\right) }{\left( n+1\right) \left( n+a-1\right) \left(
n+b-1\right) \left( n+c\right) }w_{n} \\
&\rightarrow &\frac{\left( c-b-a\right) n^{2}+\left( a+b-2ab\right)
n-c\left( a-1\right) \left( b-1\right) }{\left( n+1\right) \left(
n+a-1\right) \left( n+b-1\right) \left( n+c\right) }w_{n}=\gamma _{n}w_{n}%
\text{,}
\end{eqnarray*}%
the required recurrence relation (\ref{vn+1-n-1-rr}) follows. This completes
the proof.
\end{proof}

\section{Schur m-power convexity of a hypergeometric mean}

Schur convexity was introduced by Schur in 1923 \cite%
{Marshall-ITMA-NYAP-1979}, and it has many important applications in
analytic inequalities \cite{Aujla.369.2003, Hardy.58.1929, Zhang.126(2).1998}%
, linear regression \cite{Stepniak.36(5).1989}, graphs and matrices \cite%
{Constantine.45(2-3).1983}, combinatorial optimization \cite%
{Hwang.18(3).2004}, information-theoretic topics \cite{Forcina.42(4).1982},
Gamma functions \cite{Merkle.42(4).1982}, stochastic orderings \cite%
{Shaked.53(2).1995}, reliability \cite{Hwang.6(4).1993}, and other related
fields.

We need to recall several notions.

\begin{definition}
Let $\boldsymbol{x}=(x_{1},\ldots ,x_{n})$ and $\boldsymbol{y}=(y_{1},\ldots
,y_{n})\in \mathbb{R}^{n}$.

(i) $\boldsymbol{x}$ is said to be majorized by $\boldsymbol{y}$ (in symbols 
$\boldsymbol{x}\prec \boldsymbol{y}$) if $\sum_{i=1}^{k}x_{[i]}\leq
\sum_{i=1}^{k}y_{[i]}$ for $k=1,2,\ldots ,n-1$ and $\sum_{i=1}^{n}x_{i}=%
\sum_{i=1}^{n}y_{i}$, where $x_{[1]}\geq \dotsm \geq x_{[n]}$ and $%
y_{[1]}\geq \dotsm \geq y_{[n]}$ are rearrangements of $\boldsymbol{x}$ and $%
\boldsymbol{y}$ in a descending order.

(ii) $\Omega \subseteq \mathbb{R}^{n}$ is called a convex set if $(\alpha
x_{1}+\beta y_{1},\dotsc ,\alpha x_{n}+\beta y_{n})\in \Omega $ for any $%
\boldsymbol{x}$ and $\boldsymbol{y}\in \Omega $, where $\alpha $ and $\beta
\in \lbrack 0,1]$ with $\alpha +\beta =1$.

(iii) Let $\Omega \subseteq \mathbb{R}^{n}$, $\varphi $: $\Omega \rightarrow 
\mathbb{R}$ is said to be a Schur-convex $($Schur-concave$)$ function on $%
\Omega $ if $\boldsymbol{x}\prec \boldsymbol{y}$ on $\Omega $ \ implies $%
\varphi \left( \boldsymbol{x}\right) \leq (\geq )$ $\varphi \left( 
\boldsymbol{y}\right) $.
\end{definition}

\begin{definition}
Let $\Omega \subset \mathbb{R}^{n}(n\geq 2)$ be a set with nonempty
interior. Then $\mathbb{\phi }:\Omega \rightarrow \mathbb{R}$ is called
Schur convex on $\Omega $ if $\mathbb{\phi }(\boldsymbol{x})\leq \mathbb{%
\phi }(\boldsymbol{y})$ for each two $n$-tuples $\boldsymbol{x}%
=(x_{1},x_{2},...,x_{n})$ and $\boldsymbol{y}=(y_{1},y_{2},...,y_{n})$ of $%
\Omega $, such that $\boldsymbol{x}\prec \boldsymbol{y}$ holds. The
relationship of majorization $\boldsymbol{x}\prec \boldsymbol{y}$ means that 
\begin{equation*}
\sum_{i=1}^{k}x_{[i]}\leq \sum_{i=1}^{k}y_{[i]},\text{ \ \ }%
\sum_{i=1}^{n}x_{[i]}=\sum_{i=1}^{n}y_{[i]},\text{\ }
\end{equation*}%
where $1\leq k\leq n-1$, and $x_{[i]}$ denotes the $i$-th largest component
of $\boldsymbol{x}$. $\mathbb{\phi }$ is called Schur concave if $-\mathbb{%
\phi }$ is Schur convex.
\end{definition}

The following well-known result was proved by Marshall and Olkin \cite%
{Marshall-ITMA-NYAP-1979}.

\noindent \textbf{Theorem M-O.} \emph{Let $\Omega \subset \mathbb{R}^{2}$ be
a symmetric convex set with nonempty interior $\Omega $ and $\phi :\Omega
\rightarrow \mathbb{R}$ be a continuous and symmetric function on $\Omega $.
If $\phi $ is differentiable on $\Omega $, then $\phi $ is Schur convex
(Schur concave) on $\Omega $ if and only if%
\begin{equation*}
(y-x)\left( \frac{\partial \phi }{\partial y}-\frac{\partial \phi }{\partial
x}\right) >(<)0
\end{equation*}%
for all $(x,y)\in \Omega $ with $x\neq y$.}

In 2012, Yang \cite{Yang-PMD-1-2-2012} (see also \cite{Yang-BKM-50-2013}, 
\cite{Yang-MIA-16-2013}) extended the Schur convexity to m-power convexity.

\begin{definition}
\label{D-mpSc}Let $f:\mathbb{R}_{+}\rightarrow \mathbb{R}$ be defied by $%
f(x)=\left( x^{m}-1\right) /m$ if $m\neq 0$ and $f(x)=\ln x$ if $m=0$. The
function $\varphi :\Omega \subseteq \mathbb{R}_{+}^{n}\rightarrow \mathbb{R}$
is said to be Schur $m$-power convex (Schur $m$-power concave) on $\Omega $
if $f(\boldsymbol{x}):=(f(x_{1}),\ldots f(x_{n}))\prec f(\boldsymbol{y}%
):=(f(y_{1}),\ldots f(y_{n}))$ on $\Omega $ implies that $\varphi (%
\boldsymbol{x})\leq (\geq )\varphi (\boldsymbol{y})$.
\end{definition}

\noindent \textbf{Theorem Y.}\emph{\ Let }$U\subseteq \mathbb{R}_{+}^{n}$%
\emph{\ be a symmetric set with nonempty interior }$\mathbf{U}^{0}$\emph{\
and }$\phi :\mathbf{U}\rightarrow \mathbb{R}$\emph{\ be continuous and
differentiable in }$\mathbf{U}^{0}$\emph{. Then }$\phi $\emph{\ is Schur }$m$%
\emph{-power convex (Schur }$m$\emph{-power concave) on }$\mathbf{U}$\emph{\
if and only if }$\phi $\emph{\ is symmetric on }$\mathbf{U}$\emph{\ and }%
\begin{eqnarray}
\frac{x_{i}^{m}-x_{j}^{m}}{m}\left( x_{i}^{1-m}\frac{\partial \phi \left( 
\boldsymbol{x}\right) }{\partial x_{i}}-x_{j}^{1-m}\frac{\partial \phi
\left( \boldsymbol{x}\right) }{\partial x_{j}}\right) &\geq &(\leq )0\text{
if }m\neq 0,  \label{mp-ci} \\
(\ln x_{i}-\ln x_{j})\left( x_{i}\frac{\partial \phi \left( \boldsymbol{x}%
\right) }{\partial x_{i}}-x_{j}\frac{\partial \phi \left( \boldsymbol{x}%
\right) }{\partial x_{j}}\right) &\geq &(\leq )0\text{ if }m=0  \label{0p-ci}
\end{eqnarray}%
\emph{holds for }$\left( x_{1},x_{2},...,x_{n}\right) \in U^{0}$\emph{, }$%
i\neq j$\emph{, }$i,j=1,2,...,n$\emph{.}

\begin{remark}
Since for $x_{i},x_{j}>0$ with $x_{i}\neq x_{j}$ and 
\begin{equation*}
\frac{x_{i}^{m}-x_{j}^{m}}{m\left( x_{i}-x_{j}\right) }>0\text{ if }m\neq 0%
\text{ and }\frac{\ln x_{i}-\ln x_{j}}{x_{i}-x_{j}}>0,
\end{equation*}%
the inequalities (\ref{mp-ci}) and (\ref{0p-ci}) can be uniformly written as%
\begin{equation*}
(x_{i}-x_{j})\left( x_{i}^{1-m}\frac{\partial \phi \left( \boldsymbol{x}%
\right) }{\partial x_{i}}-x_{j}^{1-m}\frac{\partial \phi \left( \boldsymbol{x%
}\right) }{\partial x_{j}}\right) \geq (\leq )0.
\end{equation*}
\end{remark}

\begin{remark}
Putting $f(x)=x$, $\ln x$, $x^{-1}$ in Definition \ref{D-mpSc} yield the
Schur-convexity (see \cite{Marshall-ITMA-NYAP-1979}, \cite{Wang-1990}, \cite%
{Shi-TMAI-2012}), Schur-geometrically convexity (see \cite{Zhang-GCF-2004}, 
\cite{Chu-Zhang-Wang-2008}) and Schur-harmonically convexity (see \cite%
{Chu-Wang-Zhang-2011, Xia-Chu, Xia-Chu-AMS-2011}).
\end{remark}

As an application of main results, we next investigate the m-power Schur
convexity of the function $M:\mathbb{R}_{+}^{2}\rightarrow \mathbb{R}_{+}$
defined, for $a\in \left( 0,1\right) $ and $b>0$, by%
\begin{equation}
M\left( x,y\right) =\left[ \frac{1}{B\left( b,b\right) }\int_{0}^{1}\left(
sx+\left( 1-s\right) y\right) ^{a}s^{b-1}\left( 1-s\right) ^{b-1}ds\right]
^{1/a},  \label{M}
\end{equation}%
where $B\left( p,q\right) =\Gamma \left( p\right) \Gamma \left( q\right)
/\Gamma \left( p+q\right) $ is the classical Beta function.

To this end, we first note that $M\left( x,y\right) $ is a symmetric and
homogeneous mean of positive numbers $x$ and $y$. Second, assume that $y\geq
x>0$, then $M\left( x,y\right) $ has a hypergeometric series representation:%
\begin{equation}
M\left( x,y\right) =y\left[ F\left( -a,b;2b;1-x/y\right) \right] ^{1/a}\text{%
,}  \label{M-hsr}
\end{equation}%
and is called a hypergeometric mean, see \cite{Carlson-JMAA07-1963}, \cite%
{Carlson-PAMS-16-1965}, \cite{Yang-RACSAM-114-2020}. In fact, since%
\begin{equation*}
\left( sx+\left( 1-s\right) y\right) ^{a}=y^{a}\left( 1-st\right)
^{a}=y^{a}\sum_{n=0}^{\infty }\frac{\left( -a\right) _{n}}{n!}s^{n}t^{n},
\end{equation*}%
where $t=1-x/y\in \lbrack 0,1)$, we have%
\begin{eqnarray*}
M^{a}\left( x,y\right) &=&\frac{y^{a}}{B\left( b,b\right) }%
\sum_{n=0}^{\infty }\frac{\left( -a\right) _{n}}{n!}\left(
\int_{0}^{1}s^{n+b-1}\left( 1-s\right) ^{b-1}ds\right) t^{n} \\
&=&\frac{y^{a}}{B\left( b,b\right) }\sum_{n=0}^{\infty }\frac{\left(
-a\right) _{n}}{n!}B\left( n+b,b\right) t^{n}=y^{a}\sum_{n=0}^{\infty }\frac{%
\left( -a\right) _{n}\left( b\right) _{n}}{n!\left( 2b\right) _{n}}t^{n},
\end{eqnarray*}%
which implies (\ref{M-hsr}), here we have used%
\begin{equation*}
\frac{B\left( n+b,b\right) }{B\left( b,b\right) }=\frac{\Gamma \left(
n+b\right) \Gamma \left( b\right) }{\Gamma \left( n+2b\right) }\frac{\Gamma
\left( 2b\right) }{\Gamma \left( b\right) ^{2}}=\frac{\left( b\right) _{n}}{%
\left( 2b\right) _{n}}.
\end{equation*}

To obtain necessary and sufficient conditions for which $M\left( x,y\right) $
is Schur m-power convex or concave on $\mathbb{R}_{+}^{2}$, several lemmas
are needed.

\begin{lemma}
\label{L-GI}Let $0<a<a+b<1$. Then the inequality%
\begin{equation}
\frac{\Gamma \left( a+b\right) }{\Gamma \left( a+2b\right) }<\left( >\right) 
\frac{\Gamma \left( 1-a-b\right) }{\Gamma \left( 1-a\right) }  \label{GI}
\end{equation}%
holds if $a+b>\left( <\right) 1/2$.
\end{lemma}

\begin{proof}
Since $\left( \ln \Gamma \left( x\right) \right) ^{\prime \prime }=\psi
^{\prime }\left( x\right) >0$ for $x>0$, that is, $\ln \Gamma \left(
x\right) $ is convex on $\left( 0,\infty \right) $. Then for $x_{2}>\left(
<\right) x_{1}>0$ and $y_{2}>\left( <\right) y_{1}>0$, the inequality%
\begin{equation*}
\frac{\ln \Gamma \left( x_{2}\right) -\ln \Gamma \left( y_{2}\right) }{%
x_{2}-y_{2}}>\left( <\right) \frac{\ln \Gamma \left( x_{1}\right) -\ln
\Gamma \left( y_{1}\right) }{x_{1}-y_{1}}
\end{equation*}%
holds. Letting $\left( x_{2},y_{2}\right) =\left( a+2b,a+b\right) $ and $%
\left( x_{1},y_{1}\right) =\left( 1-a,1-a-b\right) $ for $a+b>\left(
<\right) 1/2$ gives the desired inequality, thereby completing the proof.
\end{proof}

By means of Corollary \ref{C-th=1} we can prove the following lemma, which
is crucial to prove Theorem \ref{T-M-SMC}.

\begin{lemma}
\label{L-Qp-m}Let $a\in \left( 0,1\right) $, $b>0$ and $p_{0}=a/\left(
2b+1\right) $. Then the function%
\begin{equation*}
Q_{p_{0}}\left( t\right) =\frac{\left( 1-t\right) ^{-p_{0}}F\left(
a,b;2b+1;t\right) }{F\left( a,b+1;2b+1;t\right) }
\end{equation*}%
is strictly decreasing (increasing) on $\left( 0,1\right) $ if $a-b>\left(
<\right) 1/2$. Consequently, the inequality%
\begin{equation}
F\left( a,b;2b+1;t\right) <\left( >\right) \left( 1-t\right) ^{p_{0}}F\left(
a,b+1;2b+1;t\right)  \label{Gm><0}
\end{equation}%
holds for $t\in \left( 0,1\right) $ if $a-b>\left( <\right) 1/2$. When $%
a-b=1/2$, $Q_{p_{0}}\left( t\right) \equiv 1$.
\end{lemma}

\begin{proof}
Let%
\begin{equation*}
\left( 1-t\right) ^{-p_{0}}F\left( a,b;2b+1;t\right) =\sum_{n=0}^{\infty
}u_{n}t^{n}.
\end{equation*}%
By Corollary \ref{C-th=1}, $u_{0}=1$, 
\begin{equation*}
u_{1}=\frac{a\left( b+1\right) }{2b+1}\text{, \ \ \ }u_{2}=\frac{1}{4}\frac{%
a\left( a+1\right) \left( b+2\right) }{2b+1},
\end{equation*}%
and for $n\geq 1$,%
\begin{equation}
u_{n+1}=2\alpha _{n}u_{n}-\beta _{n}u_{n-1},  \label{un+1-rr-a}
\end{equation}%
where%
\begin{eqnarray}
\alpha _{n} &=&\dfrac{1}{2}\dfrac{2n^{2}+\left( a+3b+2a/\left( 2b+1\right)
\right) n+ab+a}{\left( n+1\right) \left( n+2b+1\right) },  \label{an-s1} \\
\beta _{n} &=&\dfrac{\left( n+a+a/\left( 2b+1\right) -1\right) \left(
n+b+a/\left( 2b+1\right) -1\right) }{\left( n+1\right) \left( n+2b+1\right) }%
.  \label{bn-s1}
\end{eqnarray}%
It is obvious that $u_{n}>0$ for $n\geq 0$ since $p_{0},a,b>0$.

Let%
\begin{equation*}
F\left( a,b+1;2b+1;t\right) =\sum_{n=0}^{\infty }v_{n}t^{n},
\end{equation*}%
where%
\begin{equation*}
v_{n}=\frac{\left( a\right) _{n}\left( b+1\right) _{n}}{n!\left( 2b+1\right)
_{n}}=\frac{\Gamma \left( 2b+1\right) }{\Gamma \left( a\right) \Gamma \left(
b+1\right) }\frac{\Gamma \left( n+a\right) \Gamma \left( n+b+1\right) }{%
n!\Gamma \left( n+2b+1\right) }
\end{equation*}%
satisfying%
\begin{equation}
\frac{v_{n+1}}{v_{n}}=\frac{\left( n+a\right) \left( n+b+1\right) }{\left(
n+1\right) \left( n+2b+1\right) }.  \label{vn+1/n}
\end{equation}

To prove the monotonicity, we need to observe the monotonicity of the
sequence $\left\{ u_{n}/v_{n}\right\} $, which, by the monotonicity rule for
the ratio of two power series (see \cite{Biernacki-AUMC-S-9-1955}, \cite%
{Yang-JMAA-428-2015}), suffices to confirm the sign of $d_{n}=u_{n+1}-\left(
v_{n+1}/v_{n}\right) u_{n}$ for $n\geq 0$. A direct computation yields that $%
d_{0}=0$,%
\begin{equation*}
d_{1}=u_{2}-\frac{v_{2}}{v_{1}}u_{1}=\frac{1}{4}\frac{a\left( a+1\right)
\left( 2b+1\right) \left( b+2\right) }{\left( 2b+1\right) ^{2}}-\frac{\left(
1+a\right) \left( 2+b\right) }{2\left( 1+2b+1\right) }\frac{a\left(
b+1\right) }{2b+1}=0.
\end{equation*}%
Now, we establish the recurrence relation of $d_{n}$. Using the recurrence
relation (\ref{un+1-rr-a}), we obtain that, for $n\geq 1$,%
\begin{eqnarray*}
d_{n} &=&u_{n+1}-\frac{v_{n+1}}{v_{n}}u_{n}=2\alpha _{n}u_{n}-\beta
_{n}u_{n-1}-\frac{v_{n+1}}{v_{n}}u_{n} \\
&=&\left( 2\alpha _{n}-\frac{v_{n+1}}{v_{n}}\right) \left( u_{n}-\frac{v_{n}%
}{v_{n-1}}u_{n-1}\right) +\left[ \left( 2\alpha _{n}-\frac{v_{n+1}}{v_{n}}%
\right) \frac{v_{n}}{v_{n-1}}-\beta _{n}\right] u_{n-1} \\
&:&=\alpha _{n}^{\prime }d_{n-1}+\beta _{n}^{\prime }u_{n-1},
\end{eqnarray*}%
where%
\begin{equation*}
\alpha _{n}^{\prime }=2\alpha _{n}-\frac{v_{n+1}}{v_{n}}\text{ \ and \ }%
\beta _{n}^{\prime }=\alpha _{n}^{\prime }\frac{v_{n}}{v_{n-1}}-\beta _{n}.
\end{equation*}%
Substituting (\ref{an-s1}), (\ref{bn-s1}) and (\ref{vn+1/n}) into the
expressions of $\alpha _{n}^{\prime }$ and $\beta _{n}^{\prime }$ gives%
\begin{eqnarray*}
\alpha _{n}^{\prime } &=&\dfrac{2n^{2}+\left[ a+3b+2a/\left( 2b+1\right) %
\right] n+ab+a}{\left( n+1\right) \left( n+2b+1\right) } \\
&&-\frac{\left( n+a\right) \left( n+b+1\right) }{\left( n+1\right) \left(
n+2b+1\right) }=n\frac{\left( 2b+1\right) n+4b^{2}+2a-1}{\left( 2b+1\right)
\left( n+1\right) \left( n+2b+1\right) }>0,
\end{eqnarray*}%
\begin{eqnarray*}
\beta _{n}^{\prime } &=&n\frac{\left( 2b+1\right) n+4b^{2}+2a-1}{\left(
2b+1\right) \left( n+1\right) \left( n+2b+1\right) }\frac{\left(
n-1+a\right) \left( n+b\right) }{n\left( n+2b\right) } \\
&&-\dfrac{\left( n+a+a/\left( 2b+1\right) -1\right) \left( n+b+a/\left(
2b+1\right) -1\right) }{\left( n+1\right) \left( n+2b+1\right) } \\
&=&-\frac{2b\left( 2b+1-a\right) \left( a-b-1/2\right) }{\left( 2b+1\right)
^{2}}\frac{n-1}{\left( n+1\right) \left( n+2b\right) \left( n+2b+1\right) }.
\end{eqnarray*}

(i) If $a-b>1/2$, then $\beta _{n}^{\prime }<0$, which together with $%
u_{n-1}>0$ for all $n\geq 1$ yields%
\begin{equation*}
d_{n}-\alpha _{n}^{\prime }d_{n-1}=\beta _{n}^{\prime }u_{n-1}<0\text{ for }%
n\geq 1\text{.}
\end{equation*}%
Since $\alpha _{n}^{\prime }>0$ for all $n\geq 1$, it follows that%
\begin{equation*}
d_{n}<\alpha _{n}^{\prime }d_{n-1}<\alpha _{n}^{\prime }\alpha
_{n-1}^{\prime }d_{n-2}<\cdot \cdot \cdot <\alpha _{n}^{\prime }\alpha
_{n-1}^{\prime }\cdot \cdot \cdot \alpha _{1}^{\prime }d_{0}=0
\end{equation*}%
for $n\geq 1$, that is, the sequence $\left\{ u_{n}/v_{n}\right\} _{n\geq 0}$
is strictly decreasing, and so is $Q_{p_{0}}\left( t\right) $ with respect
to $t$ on $\left( 0,1\right) $.

(ii) If $a-b<1/2$, then $\beta _{n}^{\prime }>0$, which together with $%
u_{n-1}>0$ for all $n\geq 1$ yields%
\begin{equation*}
d_{n}-\alpha _{n}^{\prime }d_{n-1}=\beta _{n}^{\prime }u_{n-1}>0\text{ for }%
n\geq 1\text{.}
\end{equation*}%
It follows that%
\begin{equation*}
d_{n}>\alpha _{n}^{\prime }d_{n-1}>\alpha _{n}^{\prime }\alpha
_{n-1}^{\prime }d_{n-2}>\cdot \cdot \cdot >\alpha _{n}^{\prime }\alpha
_{n-1}^{\prime }\cdot \cdot \cdot \alpha _{1}^{\prime }d_{0}=0
\end{equation*}%
for $n\geq 1$, that is, the sequence $\left\{ u_{n}/v_{n}\right\} _{n\geq 0}$
is strictly increasing, and so is $Q_{p_{0}}\left( t\right) $ with respect
to $t$ on $\left( 0,1\right) $.

The inequality (\ref{Gm><0}) follows from the monotonicity of the function $%
Q_{p_{0}}\left( t\right) $ on $\left( 0,1\right) $, which completes the
proof.
\end{proof}

The following lemma gives in fact the necessary conditions for which $%
M\left( x,y\right) $ is Schur m-power convex or concave on $\mathbb{R}%
_{+}^{2}$.

\begin{lemma}
\label{L-Gm><0-nc}For $a,t\in \left( 0,1\right) $, $b>0$, let%
\begin{equation}
G_{m}\left( t\right) =F\left( 1-a,b;2b+1;t\right) -\left( 1-t\right)
^{1-m}F\left( 1-a,b+1;2b+1;t\right) \text{.}  \label{Gm}
\end{equation}%
(i) If $G_{m}\left( t\right) \geq 0$ for all $t\in \left( 0,1\right) $, then 
$\left( a,b,m\right) \in E^{+}$, where%
\begin{eqnarray}
E^{+} &=&\left\{ \frac{a+2b}{1+2b}-m\geq 0\right\} \cap \left( \left\{
a+b\geq 1>m\right\} \right.  \label{E+} \\
&&\left. \cup \left\{ m<a+b<1\right\} \cup \left\{ m=a+b\leq \frac{1}{2}%
\right\} \right) .  \notag
\end{eqnarray}

(ii) If $G_{m}\left( t\right) \leq 0$ for all $t\in \left( 0,1\right) $,
then $\left( a,b,m\right) \in E^{-}$, where%
\begin{eqnarray}
E^{-} &=&\left\{ \frac{a+2b}{1+2b}-m\leq 0\right\} \cap \left( \left\{
a+b\geq 1,m\geq 1\right\} \right.  \label{E-} \\
&&\left. \cup \left\{ \frac{1}{2}\leq m=a+b<1\right\} \cup \left\{
a+b<1,a+b<m\right\} \right) .  \notag
\end{eqnarray}
\end{lemma}

\begin{proof}
First, an easy computation yields%
\begin{equation*}
\lim_{t\rightarrow 0}\frac{G_{m}\left( t\right) }{t}=\frac{b\left(
1-a\right) }{2b+1}-\left[ \left( \frac{\left( 1-a\right) \left( b+1\right) }{%
2b+1}-1+m\right) \right] =\frac{a+2b}{1+2b}-m.
\end{equation*}%
Second, we compute $G_{m}\left( 1^{-}\right) $ by distinguishing three cases.

\textbf{Case 1}: $a+b>1$. We use the first formula of (\ref{F-near1}) to
compute $G_{m}\left( 1^{-}\right) $, which can be divided into three
subcases:

Subcase 1.1: If $1-m>0$, then%
\begin{equation*}
G_{m}\left( 1^{-}\right) =\frac{\Gamma \left( 2b+1\right) \Gamma \left(
a+b\right) }{\Gamma \left( a+2b\right) \Gamma \left( b+1\right) }>0.
\end{equation*}

Subcase 1.2: If $1-m=0$, then%
\begin{eqnarray*}
G_{m}\left( 1^{-}\right)  &=&\frac{\Gamma \left( 2b+1\right) \Gamma \left(
a+b\right) }{\Gamma \left( a+2b\right) \Gamma \left( b+1\right) }-\frac{%
\Gamma \left( 2b+1\right) \Gamma \left( a+b-1\right) }{\Gamma \left(
a+2b\right) \Gamma \left( b\right) } \\
&=&\left( a-1\right) \frac{\Gamma \left( 2b+1\right) \Gamma \left(
a+b-1\right) }{\Gamma \left( a+2b\right) \Gamma \left( b+1\right) }<0.
\end{eqnarray*}

Subcase 1.3: If $1-m<0$, then $G_{m}\left( 1^{-}\right) =-\infty $.

\textbf{Case 2}: $a+b=1$. We use the second formula of (\ref{F-near1}) to
compute $G_{m}\left( 1^{-}\right) $, which can be divided into two subcases:

Subcase 2.1: If $1-m>0$, then%
\begin{equation*}
G_{m}\left( 1^{-}\right) =\frac{\Gamma \left( 2b+1\right) \Gamma \left(
a+b\right) }{\Gamma \left( a+2b\right) \Gamma \left( b+1\right) }>0.
\end{equation*}

Subcase 2.2: If $1-m\leq 0$, then $G_{m}\left( 1^{-}\right) =-\infty $.

\textbf{Case 3}: $a+b<1$. Using the third formula of (\ref{F-near1}), $%
G_{m}\left( t\right) $ can be rewritten as%
\begin{equation}
G_{m}\left( t\right) =F\left( 1-a,b;2b+1;t\right) -\left( 1-t\right)
^{a+b-m}F\left( a+2b,b;2b+1;t\right) \text{.}  \label{Gm-a+b<1}
\end{equation}

Subcase 3.1: $a+b-m>0$. By the first formula of (\ref{F-near1}) we have%
\begin{equation*}
G_{m}\left( 1^{-}\right) =\frac{\Gamma \left( 2b+1\right) \Gamma \left(
a+b\right) }{\Gamma \left( a+2b\right) \Gamma \left( b+1\right) }>0.
\end{equation*}

Subcase 3.2: $a+b-m=0$. Using the first formula of (\ref{F-near1}) again,
then applying Lemma \ref{L-GI}, we have%
\begin{eqnarray*}
G_{m}\left( 1^{-}\right) &=&\frac{\Gamma \left( 2b+1\right) \Gamma \left(
a+b\right) }{\Gamma \left( a+2b\right) \Gamma \left( b+1\right) }-\frac{%
\Gamma \left( 2b+1\right) \Gamma \left( 1-a-b\right) }{\Gamma \left(
1-a\right) \Gamma \left( b+1\right) } \\
&=&\frac{\Gamma \left( 2b+1\right) }{\Gamma \left( b+1\right) }\left[ \frac{%
\Gamma \left( a+b\right) }{\Gamma \left( a+2b\right) }-\frac{\Gamma \left(
1-a-b\right) }{\Gamma \left( 1-a\right) }\right] \left\{ 
\begin{array}{cc}
\leq 0 & \text{if }a+b\geq \frac{1}{2}, \\ 
\geq 0 & \text{if }a+b\leq \frac{1}{2}.%
\end{array}%
\right.
\end{eqnarray*}

Subcase 3.3: $a+b-m<0$. We have $G_{m}\left( 1^{-}\right) =-\infty $.

Taking into account the inequality $\left( a+2b\right) /\left( 1+2b\right)
-m\geq 0$ and Subcases 1.1, 2.1, 3.1 and 3.2 with $a+b\leq 1/2$ gives that $%
\left( a,b,m\right) \in E^{+}$ if $G_{m}\left( t\right) \geq 0$ for all $%
t\in \left( 0,1\right) $. Other cases gives that $\left( a,b,m\right) \in
E^{-}$ if $G_{m}\left( t\right) \leq 0$ for all $t\in \left( 0,1\right) $.
This completes the proof.
\end{proof}

We now state and prove the following theorem.

\begin{theorem}
\label{T-M-SMC}Let $a\in \left( 0,1\right) $, $b>0$ with $a+b\geq 1/2$. Then
the mean $M\left( x,y\right) $ of $x,y>0$ defined by (\ref{M}) is Schur
m-power convex (concave) on $\mathbb{R}_{+}^{2}$ if and only if $\left(
m,a,b\right) \in E^{+}$ ($E^{-}$), where $E^{+}$ and $E^{-}$ are given by (%
\ref{E+}) and (\ref{E-}), respectively.
\end{theorem}

\begin{proof}
Since $M\left( x,y\right) $ is symmetric with respect to $x$ and $y$, we
assume that $y>x$. Let $F\equiv F\left( -a,b;2b;t\right) $, $t=1-x/y\in
\lbrack 0,1)$. Differentiation yields%
\begin{equation*}
\frac{\partial M}{\partial x}=-\frac{1}{a}F^{1/a-1}F^{\prime }\text{ \ and \ 
}\frac{\partial M}{\partial y}=F^{1/a}+\frac{1}{a}\frac{x}{y}%
F^{1/a-1}F^{\prime }.
\end{equation*}%
Then%
\begin{eqnarray*}
y^{1-m}\frac{\partial G}{\partial y}-x^{1-m}\frac{\partial G}{\partial x}
&=&y^{1-m}\left( F^{1/a}+\frac{1}{a}\frac{x}{y}F^{1/a-1}F^{\prime }\right)
+x^{1-m}\frac{1}{a}F^{1/a-1}F^{\prime } \\
&=&\frac{1}{2}y^{1-m}F^{1/a-1}\times G_{m}\left( t\right) ,
\end{eqnarray*}%
where%
\begin{equation*}
G_{m}\left( t\right) =2F+\frac{2}{a}\left( 1-t\right) F^{\prime }+\frac{2}{a}%
\left( 1-t\right) ^{1-m}F^{\prime }\text{.}
\end{equation*}

According to Theorem\textbf{\ }Y\textbf{, }$M\left( x,y\right) $ is Schur
m-power convex (concave) on $\mathbb{R}_{+}^{2}$ if and only if $G_{m}\left(
t\right) \geq \left( \leq \right) 0$ for all $t\in \left( 0,1\right) $.

We now show that $G_{m}\left( t\right) $ here is just that given by (\ref{Gm}%
). Due to%
\begin{equation*}
F^{\prime }=F^{\prime }\left( -a,b;2b;t\right) =-\frac{a}{2}F\left(
1-a,b+1;2b+1\right) ,
\end{equation*}%
we obtain%
\begin{equation*}
2F+\frac{2}{a}\left( 1-t\right) F^{\prime }=2F\left( -a,b;2b;t\right)
-\left( 1-t\right) F\left( 1-a,b+1;2b+1\right)
\end{equation*}%
\begin{eqnarray*}
&=&2\sum_{n=0}^{\infty }\frac{\left( -a\right) _{n}\left( b\right) _{n}}{%
n!\left( 2b\right) _{n}}t^{n}-\sum_{n=0}^{\infty }\frac{\left( 1-a\right)
_{n}\left( b+1\right) _{n}}{n!\left( 2b+1\right) _{n}}t^{n}+\sum_{n=0}^{%
\infty }\frac{\left( 1-a\right) _{n}\left( b+1\right) _{n}}{n!\left(
2b+1\right) _{n}}t^{n+1} \\
&=&1+\sum_{n=1}^{\infty }\left( 2\frac{n+2b}{2b}\frac{-a}{n-a}-\frac{n+b}{b}+%
\frac{n\left( n+2b\right) }{b\left( 1-a+n-1\right) }\right) \frac{\left(
1-a\right) _{n}\left( b\right) _{n}}{n!\left( 2b+1\right) _{n}}t^{n} \\
&=&\sum_{n=}^{\infty }\frac{\left( 1-a\right) _{n}\left( b\right) _{n}}{%
n!\left( 2b+1\right) _{n}}t^{n}=F\left( 1-a,b;2b+1;t\right) .
\end{eqnarray*}

The necessary conditions such that $G_{m}\left( t\right) \geq \left( \leq
\right) 0$ for all $t\in \left( 0,1\right) $ follow from Lemma \ref%
{L-Gm><0-nc}.

(i) Next we first show the sufficient conditions for which $G_{m}\left(
t\right) \geq 0$ for all $t\in \left( 0,1\right) $ by distinguishing three
cases.

Case 1.1: $\left( m,a,b\right) \in \left\{ m\leq m_{0}\right\} \cap \left\{
a+b\geq 1>m\right\} $, where $m_{0}=\left( a+2b\right) /\left( 1+2b\right) $%
. Replacing $\left( a,p_{0}\right) $ by $\left( 1-a,1-m_{0}\right) $ in
Lemma \ref{L-Qp-m} gives%
\begin{equation*}
F\left( 1-a,b;2b+1;t\right) >\left( <\right) \left( 1-t\right)
^{1-m_{0}}F\left( 1-a,b+1;2b+1;t\right)
\end{equation*}%
holds for $t\in \left( 0,1\right) $ if $\left( 1-a\right) -b<\left( >\right)
1/2$. That is, 
\begin{equation}
G_{m_{0}}\left( t\right) >\left( <\right) 0\text{ for }t\in \left(
0,1\right) \text{ if }a+b>\left( <\right) \frac{1}{2}.  \label{Gm0><0}
\end{equation}%
Now, for $m\leq m_{0}$ and $a+b\geq 1>m$, we deduce that $G_{m}\left(
t\right) \geq G_{m_{0}}\left( t\right) \geq 0$ for $t\in \left( 0,1\right) $.

Case 1.2: $\left( m,a,b\right) \in \left\{ m\leq m_{0}\right\} \cap \left\{
m<a+b<1\right\} $. This can be divided into two subcases.

Subcases 1.2.1: $m\leq m_{0}=\left( a+2b\right) /\left( 1+2b\right) \leq
a+b<1$. This implies that $a+b\geq 1/2$. Then the inequality (\ref{Gm0><0})
in combination with $m\leq m_{0}$ yields that $G_{m}\left( t\right) \geq
G_{m_{0}}\left( t\right) \geq 0$ for $t\in \left( 0,1\right) $.

Subcase 1.2.2: $m<a+b\leq m_{0}=\left( a+2b\right) /\left( 1+2b\right) <1$.
This implies that $a+b\leq 1/2$. Using another representation of $%
G_{m}\left( t\right) $, that is, (\ref{Gm-a+b<1}), we have%
\begin{eqnarray*}
G_{m}\left( t\right) &=&F\left( 1-a,b;2b+1;t\right) -\left( 1-t\right)
^{a+b-m}F\left( a+2b,b;2b+1;t\right) \\
&>&F\left( 1-a,b;2b+1;t\right) -F\left( a+2b,b;2b+1;t\right) \\
&=&\sum_{n=0}^{\infty }\left[ \left( 1-a\right) _{n}-\left( a+2b\right) _{n}%
\right] \frac{\left( b\right) _{n}}{n!\left( 2b+1\right) _{n}}t^{n}\geq 0,
\end{eqnarray*}%
where the last inequality holds since $\left( 1-a\right) _{n}\geq \left(
a+2b\right) _{n}$ which follows from $1-a+j\geq a+2b+j$ for $0\leq j\leq n-1$%
.

Case 1.3: $\left( m,a,b\right) \in \left\{ m\leq m_{0}\right\} \cap \left\{
m=a+b\leq 1/2\right\} $. In the same way as Subcase 1.2.2, we deduce that $%
G_{m}\left( t\right) \geq 0$ for $t\in \left( 0,1\right) $.

(ii) Final, we show the sufficient conditions for which $G_{m}\left(
t\right) \leq 0$ for all $t\in \left( 0,1\right) $ by distinguishing three
cases.%
\begin{eqnarray}
E^{-} &=&\left\{ \frac{a+2b}{1+2b}-m\leq 0\right\} \cap \left( \left\{
a+b\geq 1,m\geq 1\right\} \right. \\
&&\left. \cup \left\{ \frac{1}{2}\leq m=a+b<1\right\} \cup \left\{ a+b<\min
\left\{ m,1\right\} \right\} \right) .  \notag
\end{eqnarray}

Case 2.1: $\left( m,a,b\right) \in \left\{ m\geq m_{0}\right\} \cap \left\{
a+b\geq 1,m\geq 1\right\} $. We get that $G_{m}\left( t\right) \leq
G_{1}\left( t\right) <0$ since%
\begin{eqnarray*}
G_{1}\left( t\right) &=&F\left( 1-a,b;2b+1;t\right) -F\left(
1-a,b+1;2b+1;t\right) \\
&=&\sum_{n=0}^{\infty }\left[ \left( b\right) _{n}-\left( b+1\right) _{n}%
\right] \frac{\left( 1-a\right) _{n}}{n!\left( 2b+1\right) _{n}}t^{n}<0\text{%
.}
\end{eqnarray*}

Case 2.2: $\left( m,a,b\right) \in \left\{ m\geq m_{0}\right\} \cap \left\{
1/2\leq m=a+b<1\right\} $. Using another representation of $G_{m}\left(
t\right) $, that is, (\ref{Gm-a+b<1}), we have%
\begin{eqnarray*}
G_{m}\left( t\right) &=&F\left( 1-a,b;2b+1;t\right) -F\left(
a+2b,b;2b+1;t\right) \\
&=&\sum_{n=0}^{\infty }\left[ \left( 1-a\right) _{n}-\left( a+2b\right) _{n}%
\right] \frac{\left( b\right) _{n}}{n!\left( 2b+1\right) _{n}}t^{n}\leq 0,
\end{eqnarray*}%
where the inequality hold due to $\left( 1-a\right) _{n}\leq \left(
a+2b\right) _{n}$ which follows from $a+b\geq 1/2$.

Case 2.3: $\left( m,a,b\right) \in \left\{ m\geq m_{0}\right\} \cap \left\{
a+b<1,a+b<m\right\} $. This can be divided into three subcases.

Subcase 2.3.1: $\left( m,a,b\right) \in \left\{ m\geq m_{0}\right\} \cap
\left\{ a+b<1\leq m\right\} $. As shown in Case 2.1, we have $G_{m}\left(
t\right) \leq G_{1}\left( t\right) \leq 0$.

Subcase 2.3.2: $\left( m,a,b\right) \in \left\{ m_{0}\leq a+b<m<1\right\} $.
This implies that $a+b\geq 1/2$. Then by (\ref{Gm-a+b<1}), we have%
\begin{eqnarray*}
G_{m}\left( t\right) &<&F\left( 1-a,b;2b+1;t\right) -F\left(
a+2b,b;2b+1;t\right) \\
&=&\sum_{n=0}^{\infty }\left[ \left( 1-a\right) _{n}-\left( a+2b\right) _{n}%
\right] \frac{\left( b\right) _{n}}{n!\left( 2b+1\right) _{n}}t^{n}\leq 0,
\end{eqnarray*}%
where the inequality hold due to $\left( 1-a\right) _{n}\leq \left(
a+2b\right) _{n}$ which follows from $a+b\geq 1/2$.

Subcase 2.3.3: $\left( m,a,b\right) \in \left\{ a+b<m_{0}\leq m<1\right\} $.
This implies that $a+b\leq 1/2$. By the inequality (\ref{Gm0><0}) we acquire
that $G_{m}\left( t\right) \leq G_{m_{0}}\left( t\right) \leq 0$.

This completes the proof.
\end{proof}

\end{document}